\documentclass[12pt]{amsart}

\usepackage{graphicx, amsfonts, amssymb, amsthm, amsmath, mathrsfs, latexsym,stmaryrd,holtpolt,commath}
\usepackage{caption}
\usepackage{subcaption}
\usepackage{tikz-cd}

\usetikzlibrary{arrows}
\let\amsamp=&

\usepackage[all,cmtip]{xy}
\usepackage{mathrsfs,holtpolt}
\usepackage{xcolor}
\usepackage{verbatim}
\usepackage[shortlabels]{enumitem}

\usepackage{geometry}
\allowdisplaybreaks

\newtheorem{thm}{Theorem}[section]         
\newtheorem{prop}[thm]{Proposition}

\theoremstyle{definition}

\numberwithin{equation}{section}

\renewcommand{\tilde}[1]{\widetilde{#1}}%

\renewcommand{\AA}{\mathcal{A}}
\newcommand{\BB}{\mathcal{B}}

\renewcommand{\SS}{\mathcal{S}}
\renewcommand{\SS}{\mathcal{S}}
\newcommand{\LL}{\mathcal L}
\newcommand{\MM}{\mathcal M}

\newcommand{\TT}{\mathcal T}

\newcommand{\N}{\mathbb N}
\newcommand{\Q}{\mathbb Q}
\newcommand{\Z}{\mathbb Z}

\newcommand{\R}{\mathbb R}
\renewcommand{\H}{\mathbb{H}}

\newcommand{\wS}{\TT}
\newcommand{\wA}{\BB}

\newcommand{\F}{\mathbb{F}}
\renewcommand{\Re}{\operatorname{Re}}

\newcommand{\llangle}{\langle\! \langle}
\newcommand{\rrangle}{\rangle\! \rangle}
\newcommand{\tT}{\tilde{T}}

\newcommand{\tOmega}{\tilde{\Omega}}
\newcommand{\wRhoe}{\sigma_e}

\newcommand{\gplus}{\gamma_{\infty}}
\newcommand{\gneg}{\gamma_{-\infty}}

\newcommand{\e}{\epsilon}

\newcommand{\x}{\xi_\gamma}
\newcommand{\y}{\eta_\gamma}
\newcommand{\PSL}{\operatorname{PSL}(2,\Z)}
\newcommand{\sgn}{\operatorname{sign}}

\title{Geodesic flows and slow downs of continued fraction maps}
\author{Claire Merriman}

\begin{document}
\begin{abstract}
    The connection between cutting sequences of geodesics on the modular surface $\PSL\backslash\H$ and regular continued fractions was established by Series. Heersink expanded the cross-section of the geodesic flow on the unit tangent bundle to the modular surface to describe the Farey tent-map as a slowdown of the Gauss map for the regular continued fractions. Boca and the author expanded the connection between cutting sequences of geodesics on the modular surface $\Theta\backslash\H$ and even continued fractions, which was previously established as a billiard flow by Bauer and Lopes. We will similarly expand the cross-section of the geodesic flow on this unit tangent bundle to describe the three-branch slowdown of the even Farey map. 
\end{abstract}

\maketitle

\section{Introduction}\label{intro}

Caroline Series \cite{Ser} established explicit connections between the geodesics on $\MM=\PSL\backslash\H$, geodesic codings, and regular continued fraction dynamics. The connection between geodesics on the modular surface $\MM$ and continued fractions was previously established by Artin \cite{Art} who used continued fractions to show the existence of dense geodesics on $\MM$. Series' construction defines a cross-section of the geodesic flow on $T^1\MM$ whose first return map provides a double cover of the natural extension of the Gauss map, as well as defining a cutting sequence which explicitly describes the relationship between between geodesics on $\MM$ and the (regular) continued fraction expansion of the endpoints of the lifts of the geodesic to $\H$. Heersink \cite{Hee} expanded the cross-section to provide a double cover of the natural extension of the Farey map, which is a slowdown of the Gauss map. 

Along with Boca, the author established similar connections between modular surfaces, geodesic codings, the even continued fractions, and the odd continued fractions \cite{BM1}. The even continued fractions provide a classification of geodesics on 
$\MM_e=\Theta\backslash\H$, where 
\[\Theta=\{\left(\begin{smallmatrix}
    a & b\\ c& d
\end{smallmatrix}\right)
\in\PSL:a\equiv d\!\pmod 2,\ b\equiv c\!\pmod 2,\ a\not\equiv b\!\pmod 2\}.\]
The $\Theta-$group was previous used to describe the even continued fractions as a billiard flow by Bauer and Lopes \cite{BL}, while Kraaikamp and Lopes \cite{KL} use the even Gauss map to prove the asymptotic growth of the length of trajectories of the $\Theta-$group. 

The goal of this paper is to study the slowdown of the even Gauss map, which we will call the even Farey map, and then expand the cross-section of the geodesic flow on $T^1\MM_e$ so that the first-return map is a double cover of the natural extension of the even Farey map. We will then find an invariant measure of the natural extension of the even Farey map and show that this gives the invariant measure of the even Farey map from , and prove that these systems are ergodic. 

The even Farey map has been studied in context of the even Gauss map by Boca and Linden \cite{BocaL}. It has also appeared in the study of
Pythagorean triples by Romik \cite{Rom}, and by Aaronson and Denker \cite{AaD} as a shift on the intervals $(0,\tfrac{1}{3}),(\tfrac{1}{3},\tfrac{1}{2}),(\tfrac{1}{2},1)$. Both papers realize the map as a factor of the first return map to a cross-section of the geodesic flow on the unit tangent bundle to $\Gamma(2)\backslash\H$, where $\Gamma(2)=\{M\in\PSL : M\equiv I_2 \pmod 2\}$ is an index 2 subgroup of $\Theta$. However, the natural extension of the even Farey map and the use of the ergodicity of the geodesic flow to establish ergodic properties are new.

In Section \ref{series}, we provide an overview of the dynamics of the regular continued fractions and the Gauss map, as well as a detailed summary of Series' cutting sequence construction. We also give an example of the cutting sequence on the modular surface in Figure \ref{modboat_figure}, further clarifying the connection between the cutting sequence and the modular surface. We believe that the image of the cutting sequences on the modular surface is new. In Section \ref{farey}, we introduce the Farey map and Heersink's cross-section. We also give an explicit description of the M\"obius transformation on $\H$ which induces the first return to the cross-section of the geodesic flow on $T^1\MM$ and of the action on the cutting sequences and the regular continued fraction expansion of the endpoints of the geodesics. 

In Section \ref{ecf_gauss}, we define the even continued fraction Gauss map and its natural extension, define a new tessellation of $\H$, and modify the cutting sequence description from \cite{BM1} to more closely correspond to the behavior of geodesics on $\MM_e$.

In Section \ref{ecf_farey}, we define the even Farey map and its natural extension. We then expand the cross-section from Section \ref{ecf_gauss}, realizing the first return map as a map induced by M\"obius transformations on $\H$. Section \ref{diagram} provides a relationship between the construction in Section \ref{farey} and the author's description of the slow continued fractions using similar methods \cite{Mer}.

\section{Farey map and regular continued fractions}\label{regular}
\subsection{Regular continued fractions and cutting sequences}\label{series}
The regular continued fraction expansion of $x>0$ is
\begin{equation*}
x= n_0+\cfrac{1}{n_1+\cfrac{1}{n_2+\dots}}
=[n_0;n_1,n_2,\ldots],\quad n_i\in\Z,\ n_0\geq 0,\ n_i>0 \textnormal{ for } i>0.
\end{equation*}
The regular continued fraction expansion of irrational numbers is unique, and there are two expansions of rational numbers coming from the fact that $n=n-1+\frac{1}{1}$.

The \emph{Gauss map} $G:[0,1)\to[0,1)$ given by \[
    G(x)=\begin{cases}
        \frac{1}{x}-\left\lfloor\frac{1}{x}\right\rfloor, & x\neq0 \\
        0, & x=0
    \end{cases}=\begin{cases}
        \frac{1}{x}-k, & x\in\left(\frac{1}{k+1},\frac{1}{k}\right] \\
        0, & x=0,
    \end{cases}\]
generates the regular continued fraction expansion of $x\in[0,1)$, and 
\[G[0;n_1,n_2,n_3,\dots]=[0;n_2,n_3,\dots].\]

The natural extension of the Gauss map is a two dimensional, invertible extension of the Gauss map. For the regular continued fractions, the natural extension is $\overline{G}:[0,1)^2\to[0,1)^2$ by 
\[
    \overline{G}(x,y)=\begin{cases}
        \left(\frac{1}{x}-\left\lfloor\frac{1}{x}\right\rfloor,\frac{1}{\left\lfloor\frac{1}{x}\right\rfloor+y}\right), & x\neq0 \\
        (0,y), & x=0
    \end{cases}=\begin{cases}
        \left(\frac{1}{x}-k,\frac{1}{k+y}\right), & x\in\left(\frac{1}{k+1},\frac{1}{k}\right] \\
        (0,y), & x=0,
    \end{cases}\]
and \[\overline{G}([0;n_1,n_2,n_3\dots],[0;n_0,n_{-1},\dots])=([0;n_2,n_3,\dots],[0;n_1,n_0,n_{-1},\dots]).\]
Series \cite{Ser} gave an explicit construction of the connection between regular continued fraction expansions of real numbers, cutting sequences of geodesics on the upper half plane, and the geodesic flow on the unit tangent bundle of the modular surface. Since we will be modifying this construction, we start with a summary of Series' results.

First, define the \emph{Farey tessellating $\F$} on the upper half plane $\H$ as the tessellation whose edges are given by $i\R$ under $\PSL$ acting by M\"obius transformations. The fundamental Farey cell is the ideal triangle $\{0,1,\infty\},$ which is a threefold cover of the nonstandard fundamental Dirichlet region 
$\mathfrak{F}=\{z\in\H:0\leq \operatorname{Re}z\leq 1, |z-\tfrac{1}{2}|\leq\tfrac{1}{2}\}.$
 Let $\MM=\PSL\backslash\H$, which is a trice punctured sphere with a cusp at $\pi(\infty)$ and cone points at $\pi(i)$ and $\pi(\frac{1+i\sqrt{3}}{2})$. Then the edges of the Farey tessellation project to the line $S$ which runs from the cusp $\pi(\infty)$  to $\pi(i)$ and back.

Let $\SS=\pm\big((1,\infty)\times(-1,0)\big)$, and let $\AA$ be the set of geodesics in $\H$ with endpoints $(\gplus,\gneg)\in\SS$.
An oriented geodesic $\gamma$ is cut into segments by the ideal triangles in $\F$. The two sides of an ideal triangle which cut $\gamma$ meet at a vertex that is either on the right or the left of $\gamma$, as in Figure \ref{farey_tes}. We label the corresponding segment of $\gamma$ with $R$ or $L$, respectively. On $\MM$, the cusp is on the right of the corresponding segment of the oriented geodesic $\overline{\gamma}=\pi(\gamma)$ as it crosses $S$ if the segment is labeled $R$ and on the left if the segment is labeled $L$ as in Figure \ref{modboat_figure}.

\begin{figure}
    \begin{subfigure}{.45\textwidth}
        \centering
        \includegraphics[width=.45\textwidth]{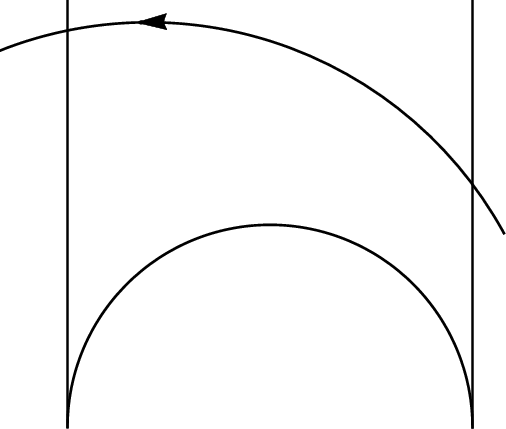}
        \includegraphics[width=.45\textwidth]{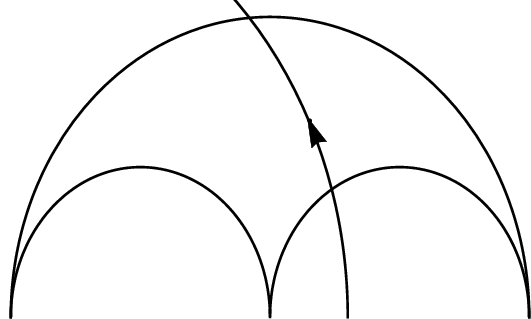}
        \caption{Geodesics cut by two sides that meet on the right of the oriented geodesic}
    \end{subfigure}%
    \hfill
    \begin{subfigure}{.45\textwidth}
        \centering
        \includegraphics[width=.45\textwidth]{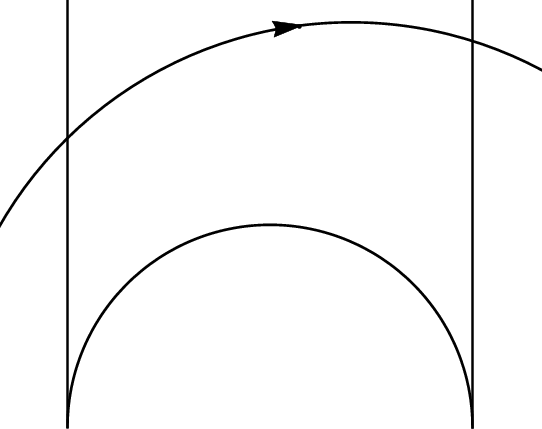}
        \includegraphics[width=.45\textwidth]{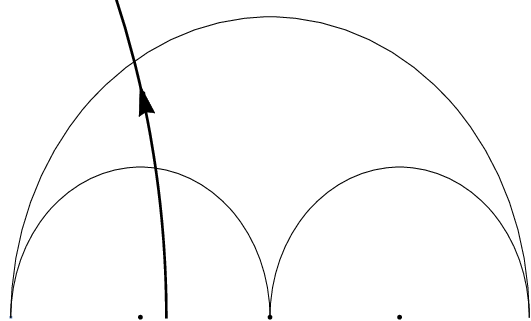}
        \caption{Geodesics cut by two sides that meet on the left of the oriented geodesic}
    \end{subfigure}
    \caption{Examples of geodesic segments cut by edges of the Farey tessellation  $\F$}
    \label{farey_tes}
\end{figure}

\begin{figure}
    \includegraphics[width=.55\textwidth]{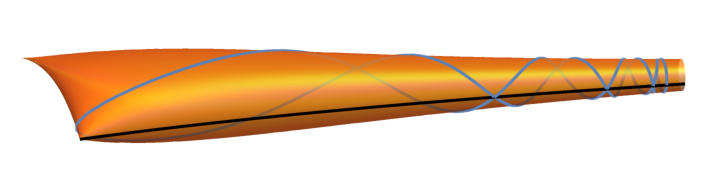}
        \includegraphics[width=.4\textwidth]{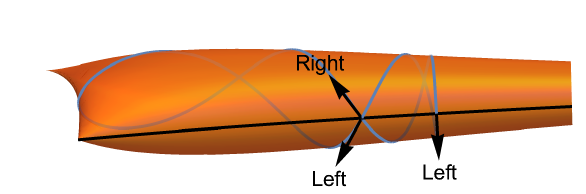}
    \caption{Geodesics on the modular surface $\MM$, shown in blue, with the image of $i\R$ shown in black. The unit tangent vectors in the second image are labeled based on whether the cusp is on the left or right.}
    \label{modboat_figure}
\end{figure}

Since geodesics on the upper half plane are uniquely defined by their endpoints, for every M\"obius transformation $\rho$ leaving $\SS$ invariant, we also use $\rho$ for the map induced on $\AA$. To every geodesic $\gamma \in \AA$ we associate the positively oriented geodesic arc $[\xi_\gamma,\eta_\gamma]$, where $\xi_\gamma$ and $\eta_\gamma$ are defined by
$\xi_\gamma =
 \gamma \cap i\R,
\eta_\gamma =\gamma \cap \lfloor\gplus\rfloor+i\R.$ 
When $\gplus\geq 1$, the positively oriented geodesic $[\xi_\gamma,\eta_\gamma]$ crosses $\lfloor\gplus\rfloor$ ideal triangles with edges that meet at infinity on the left. When $\gplus\leq -1,$ the $\lfloor\gplus\rfloor$  ideal triangles have edges that meet on the right. Thus, the cutting sequences $\x L^{n_0}\y$ and $\x R^{n_0}\y$ indicate that the first digit of the continued fraction expansion of $|\gplus|$ is $n_0$. Modifying the second part of Series' Theorem A to include the map $\rho$ defined in the proof, we get

\begin{thm}\cite[Theorem A]{Ser} 
    Let $\rho:\SS\mapsto\SS$ with 
    $(x,y)\mapsto \left(\frac{1}{\e \lfloor x\rfloor-x},\frac{1}{\e \lfloor x\rfloor-y}\right),$ where $\e=\sgn(x)$. Then $\rho(\y)=\xi_{\rho(\gamma)}$.

    If $(\gplus,\gneg)$ defines a geodesic with cutting sequence $\dots L^{n_{-2}}R^{n_{-1}}\x L^{n_0}\y R^{n_1}L^{n_2}\dots$, then 
    \[\gplus=[n_0;n_1,n_2,\dots],\quad \gneg=-[0;n_{-1},n_{-2},\dots],\]
    and $\rho(\gplus,\gneg)$ defines a geodesic with cutting sequence $\dots L^{n_{-2}}R^{n_{-1}} L^{n_0}\xi_{\rho(\gamma)}R^{n_1}\eta_{\rho(\gamma)}L^{n_2}\dots.$ 

    If $(\gplus,\gneg)$ defines a geodesic with cutting sequence $\dots R^{n_{-2}}L^{n_{-1}}\x R^{n_0}\y L^{n_1}R^{n_2}\dots$, then 
    \[\gplus=-[n_0;n_1,n_2,\dots],\quad \gneg=[0;n_{-1},n_{-2},\dots],\]
    and $\rho(\gplus,\gneg)$ defines a geodesic with cutting sequence $\dots R^{n_{-2}}L^{n_{-1}} R^{n_0}\xi_{\rho(\gamma)}L^{n_1}\eta_{\rho(\gamma)}R^{n_2}\dots.$ 
\end{thm}
Similar to the Gauss map, if $\gplus=\e[n_0;n_1,n_2,\dots],\gneg=-[0;n_{-1},n_{-2},\dots], \e=\pm1$, then 
\[\rho(\e[n_0;n_1,n_2,\dots],-\e[0;n_{-1},n_{-2},\dots])=(-\e[n_1;n_2,\dots],\e[0;,_0,n_{-1},n_{-2},\dots]).\]

We also define $u_\gamma$ to be the unit tangent vector based at $\x$ pointing along $\gamma$, and $X$ the set of unit tangent vectors in $T^1\MM$ based at $x\in S$ which points along a geodesic which changes type at $x$. Since every geodesic $\overline{\gamma}$ on $\MM$, other than $S$ has a lift in $\AA$, $X=\{\pi(\x,u_\gamma):\textnormal{cutting sequence changes type and $\pi(\x)$}\}$. Series proved that the map $i:\AA\to X$ where $i(\gamma)=\pi(\x,u_\gamma)$ is continuous, open, and bijective \cite[Theorem A]{Ser}
\footnote{Series' map is injective except for the oppositely oriented geodesics between $+1$ and $-1$ which project to the same line, but we instead removed that measure zero set from the definition of $\SS$.}. 
Since geodesics on the upper half plane are also uniquely defined by a unit tangent vector and a base point, $\rho$ induces the first return map $P$ on $X$.

Finally, we summarize the rest of the maps in Series' Section 2 in the following diagram 

\begin{center}
    \begin{tikzcd}[ampersand replacement=\&,column sep=2em,row sep=3em]
        X \arrow[d, "P"] 
        \&                         
            \& \AA \arrow[ll, "i"'] \arrow[d, "\rho"] 
                \& \SS \arrow[l, Rightarrow, no head, swap, "\sim"]  
                \arrow[rr, "{\left(\frac{1}{|x|},|y|\right)}"]
                \arrow[d,"\rho"] 
                \&  \&{[0,1)^2} \arrow[d,"\overline{G}"] \\
        X   \&  \& \AA \arrow[ll, "i"]                 
            \& \SS \arrow[l, Rightarrow, no head, swap, "\sim"] 
                \arrow[rr,"{\left(\frac{1}{|x|},|y|\right)}"]
                \&  \& {[0,1)^2}.        
        \end{tikzcd}
\end{center}
Since $\frac{d\alpha d\beta d\theta}{(\alpha-\beta)^2}$ is the invariant measure for the first return map to a cross-section of the geodesic flow on $T^1\MM$, $\frac{d\alpha d\beta}{(\alpha-\beta)^2}$ is the invariant measure for $\rho$. Let $\e=\sgn(x)$. Using the projections $J:\SS\to[0,1)^2\times\{\pm1\}$ by 
$J(x,y)=\left(\frac{\e}{x},-\e y,\e\right)$ and $(x,y,\e)\mapsto(x,y)$, gives 
\[\frac{d\tfrac{\e}{x}d(-\e y)}{(\tfrac{\e}{x}-(-\e y))^2}
=\frac{dxdy}{(1+xy)^2}\]
    is the invariant measure for $\overline{G}$.

\subsection{The Farey map}\label{farey}
The Farey map $F:[0,1)\to[0,1)$ by \[
    F(x)=\begin{cases}
        \frac{x}{1-x}, & x\in\left[0,\frac{1}{2}\right),\\
        \frac{1-x}{x}, & x\in\left[\frac{1}{2},1\right),
    \end{cases}\]
is a \emph{slowdown} of the Gauss map for the regular continued fractions, since for all $x\in[0,1)$, $F^k(x)=G(x)$ where $k=\lfloor\frac{1}{x}\rfloor$. 
On the continued fraction expansion of $x,$ 
\[F[0;n_1,n_2,n_3,\dots]=\begin{cases}
    [0;n_1-1,n_2,n_3\dots], & n_1>1,\\
    [0;n_2,n_3,\dots], & n_1=1.
\end{cases}\]
The natural extension of the Farey map is 
$\overline{F}:[0,1)^2\to[0,1)^2$ where 
\[\overline{F}(x,y)=\begin{cases}
    \left(\frac{x}{1-x},\frac{y}{1+y}\right), & x\in\left[0,\frac{1}{2}\right),\\
    \left(\frac{1-x}{x},\frac{1}{1+y}\right) & x\in\left[\frac{1}{2},1\right),
\end{cases}\]
and 
\[\overline{F}([0;n_1,n_2,n_3\dots],[0;n_0,n_{-1},\dots])=\begin{cases}
    ([0;n_1-1,n_2,n_3\dots],[0;n_0+1,n_{-1},\dots]), & n_1>1,\\
    ([0;n_2,n_3,\dots],[0;n_1,n_0,n_{-1},\dots]), & n_1=1.
\end{cases}
\]

Heersink modified Series' construction to realize the natural extension of the Farey as the projection of a M\"obius transformation acting on the endpoints of a geodesic on $\H$ \cite{Hee}. 
First, we modify $\SS$ to $\TT=\pm\bigl((1,\infty)\times(-\infty,0)\bigr)$. Then $\BB$ is the set of geodesics on $\H$ with endpoints in $\TT$, and $Y$ is the set of unit tangent vectors based on $S$. 

While not stated in \cite{Hee}, we can also modify $\rho$ to act on the cutting sequences and endpoints of $(\gplus,\gneg)$. Let $\sigma:\TT\to\TT$ by 
\[\sigma(x,y)=\begin{cases}
    (x-\e,y-\e), & \e  x\in (1,2],\\
    \left(\frac{1}{\e -x},\frac{1}{\e -y}\right), & \e x\in (2,\infty).
\end{cases}\]
In $\H$, $\x=\gamma\cap i\R$ as before, but now \[\y=\begin{cases}
    \gamma\cap \pm1+ i\R, & |\gplus|\in(2,\infty),\\
    \gamma\cap \frac{3}{2}+\frac{1}{2}e^{it}, & \gplus\in(1,2],\\
    \gamma\cap \frac{-3}{2}+\frac{1}{2}e^{it}, & \gplus\in [-2,-1).
\end{cases}\] Then $\sigma$ induces the first return map on $Y$ given in \cite{Hee}. Projecting by 
$(x,y)\mapsto\left(\frac{\e}{x},\frac{1}{1-\e y},\e\right)$ and $(x,y,\e)\mapsto(x,y)$ gives 
\[\frac{d\tfrac{\e}{x}d(\frac{-\e}{y}+\e)}{(\tfrac{\e}{x}-(\frac{-\e}{y}+\e))^2}
=\frac{dxdy}{(x+y-xy)^2}\]
is the invariant measure for $\overline{F}$.

Returning to the regular continued fraction expansion of $(\gplus,\gneg)$,  
\begin{align*}
    \gplus&=\e\left(n_0+\frac{1}{n_1+\dots}\right),\quad 
    \gneg=-\e\left(n_{-1}+\frac{1}{n_{-2}+\dots}\right),\\
    \sigma(\gplus,\gneg)&=\begin{cases}
        \left(\e\left(n_0-1+\frac{1}{n_1+\dots}\right),
        -\e\left(n_{-1}+1+\frac{1}{n_{-2}+\dots}\right)\right), & n_0>1,\\
        \left(-\e\left(n_1+\frac{1}{n_2+\dots}\right),
        \e\left(\frac{1}{n_{-1}+1+\frac{1}{n_{-2}+\dots}}\right)\right), & n_0=1.
    \end{cases}
\end{align*}
When $\gneg\in\e(-1,0),$ $n_{-1}=0,$ so $\sigma$ has the same effect as the Farey map on the regular continued fraction expansion of $(\gplus,\gneg)$.

\section{Even continued fractions and cutting sequences}\label{ecf_gauss}
\subsection{Even continued fractions}\label{intro_ecf}

Schweiger \cite{Sch1} defined the even continued fractions of $x\in[0,1]$ as
\begin{equation*}
x= \cfrac{1}{a_1+\cfrac{\e_1}{a_2+\cfrac{\e_2}{a_3+\dots}}}=[\![(a_1,\e_1) (a_2, \e_2) (a_3, \e_3) \ldots]\!],
\end{equation*}
where $\e_i =\e_i(x)\in \{\pm1\}$ and every $a_i$ is an even positive integer. 

\begin{figure}
\centering
\includegraphics[width=.5\textwidth]{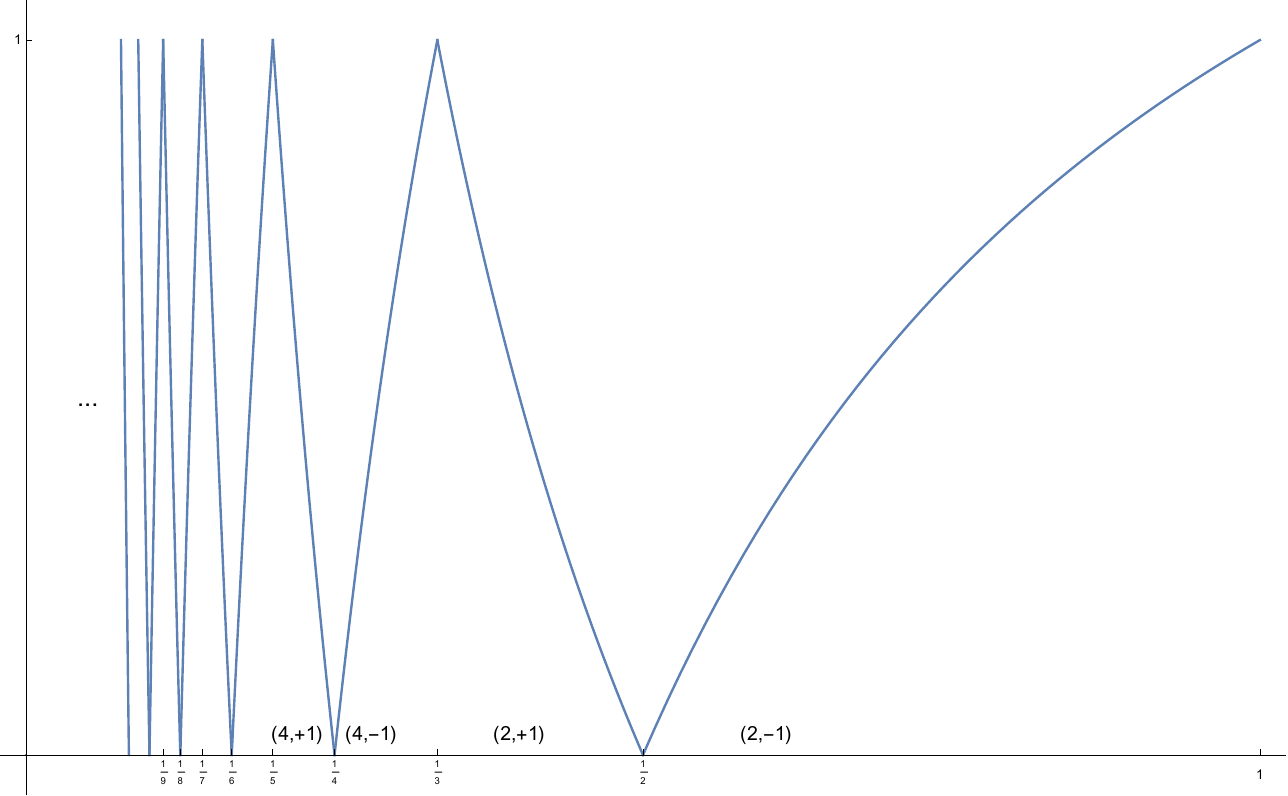}
\caption{Even Gauss Map}
\label{intro_egauss}
\end{figure}

The even Gauss map $T_e:[0,1]\to[0,1]$ (Figure \ref{intro_egauss}) is given by
\begin{equation*}
 T_e (x)= \begin{cases}
        \left|\frac{1}{x}-\left\lfloor\frac{1}{2x}+\frac{1}{2} \right\rfloor\right| 
            &\textnormal{ if } x\neq 0,\\
        0   &\textnormal{ if } x= 0
    \end{cases}=
    \begin{cases} 
        \frac{1}{x}-2k  &\textnormal{ if } x\in \left(\frac{1}{2k+1},\frac{1}{2k}\right],\\
        -\frac{1}{x}+2k  &\textnormal{ if } x\in \left(\frac{1}{2k},\frac{1}{2k-1}\right],\\
        0&\textnormal{ if } x=0.
\end{cases}\ k\geq 1.
\end{equation*}
Then $a_1(x)=2k$ and $\e_1(x)=\sgn\left(\tfrac{1}{x}-2k\right)$.
If we plug in the even continued fraction expansion to $T_e$, we again delete the first digit $(a_1,\e_1)$ of $x$, i.e.
\begin{equation*}
T_e \left( [\![(a_1,\e_1) (a_2,\e_2) (a_3,\e_3)\ldots]\!] \right) =[\![ (a_2,\e_2) (a_3,\e_3) (a_4,\e_4)\ldots ]\!] .
\end{equation*}

As with the regular continued fractions, we expand the definition of even continued fractions to all real numbers. For $x>1$, the first digit $(a_0,\e_0)=(2k,+1)$ for $x\in[2k,2k+1)$, and $(a_0,\e_0)=(2k,-1)$ for $x\in[2k-1,2k)$. We then write 
\[x=a_0+\cfrac{\e_0}{a_1+\cfrac{\e_1}{a_2+\dots}}=[\![(a_0,\e_0);(a_1,\e_1)\dots]\!].
\]

The dual continued fraction expansion to the even continued fractions are the extended even continued fractions. While we will not use this expansion for the slowdown of the Gauss map on $[0,1)$, we will use it for the backwards endpoints of geodesics on the upper half plane.
For $y\in[-1,1]$, the extended even continued fraction expansion is 
\begin{equation*}\begin{split}
y:=
\cfrac{\e_0}{b_0+\cfrac{\e_1}{b_1+\cfrac{\e_2}{b_2+\dots}}}&
=\llangle(\e_0/b_0)(\e_1/b_1)(\e_2/b_2)\dots\rrangle
\end{split}
\end{equation*}
We see that the even continued fractions are self-dual, as the extended even continued fractions are essentially a reindexing of the even continued fractions. That is, we relabel each $\e_i$ as $\e_{i+1}$. However, the map that accomplishes this reindexing is metrically complicated. 


Now we define $\overline{T}_e$ on $[0,1]\times[-1,1]$ to be 
\begin{equation}\label{tebar}
\overline{T}_e (x,y)=
\begin{cases}
 \big(  \e_1(x)( \frac{1}{x}-a_1(x)),\frac{\e_1(x)}{a_1(x)+y}\big)=(T_e(x),\frac{\e_1(x)}{a_1(x)+y})&\textnormal{if } x\neq0,\\
 (0,y)&\textnormal{if } x=0.
\end{cases}
\end{equation}
As with the regular continued fractions, we get the two-sided shift map
\begin{equation*}
\begin{split}
\overline T_e &  \big([\![(a_1,\e_1) (a_2,\e_2)\ldots ]\!] ,\llangle(\e_0/b_0)(\e_1/b_1)(\e_2/b_2)\dots\rrangle \big)
 \\&= \big([\![(a_2,\e_2) (a_3,\e_3) \ldots ]\!] ,\llangle(\e_1/a_1)(\e_0/b_0)\dots\rrangle \big).
\end{split}
\end{equation*}

\subsection{A new coding of geodesics on some modular surfaces and 
the action on the upper half plane}\label{evenoddcut}
Bauer and Lopes \cite{BL}  realized $\overline T_e$
 as a section of the billiard flow on 
\begin{equation*}
\Theta :=\left\{ M\in \Gamma (1) : M \equiv I_2 \ \mbox{\rm or}\ \left( \begin{matrix}
0 & 1 \\ 1 & 0 \end{matrix}\right) \pmod{2}  \right\}.
\end{equation*} 
Along with Boca, the author modified Bauer and Lopes' construction to describe $\overline T_e$ as a cross-section of  the geodesic flow on the modular surface  $\MM_e:=\Theta\backslash\H$ using a Series-style coding in \cite{BM1}.
Short and Walker \cite{SW16} also describe the even continued fraction expansions of rational numbers using paths along the \emph{Farey tree}. The Farey tree is defined as the geodesics in the Farey tessellation connecting two elements of the orbit $\Theta\infty=\{\frac{m}{n}\in\Q: \textnormal{ $m$ or $n$ is even}\}$. We will modify the construction in \cite{BM1} to more closely describe the relationship between geodesics on $\MM_e$ and $\H$ by removing the Farey tree. This corresponds to tessellating $\H$ by the standard Dirichlet region $\{z\in\H:|\operatorname{Re}|<1,|z|\geq 1\}$.

As in the regular continued fraction cases, we identify a geodesic in $\H$ by its endpoints $(\gplus,\gneg)$. Let $\AA_e$ is the set of geodesics in $\H$ with endpoints
\begin{equation*}
(\gplus,\gneg) \in \SS_e := \big( (-\infty ,-1) \cup (1,\infty)\big) \times (-1,1) .
\end{equation*}
Thus, for every M\"obius transformation $\rho$ leaving $\SS_e$ invariant, we also use $\rho$ for the map induced on $\AA_e$. To every geodesic $\gamma \in \AA_e$ we associate the positively oriented geodesic arc $[\xi_\gamma,\eta_\gamma]$, where $\xi_\gamma$ and $\eta_\gamma$ are defined by
\begin{equation*}
\xi_\gamma =\begin{cases} \gamma \cap 1+i\R & \mbox{\rm if $\gamma_\infty >1$}, \\
\gamma \cap -1+i\R & \mbox{\rm if $\gamma_\infty < -1$,} \end{cases}
\qquad
\eta_\gamma =\begin{cases} \gamma \cap a_1+\frac{\e_1}{2}+\frac{1}{2}e^{it} & \mbox{\rm if $\gamma_\infty >1$}, \\
\gamma \cap -a_1-\frac{\e_1}{2}+\frac{1}{2}e^{it} & \mbox{\rm if $\gamma_\infty < -1$.} \end{cases}
\end{equation*}
Every geodesic $\overline{\gamma}$ on $\MM_e$ lifts to $\H$ to a geodesic $\gamma \in \AA_e$ \cite{BM1}.

\begin{figure}
    \begin{subfigure}{.45\textwidth}
        \includegraphics[width=\textwidth]{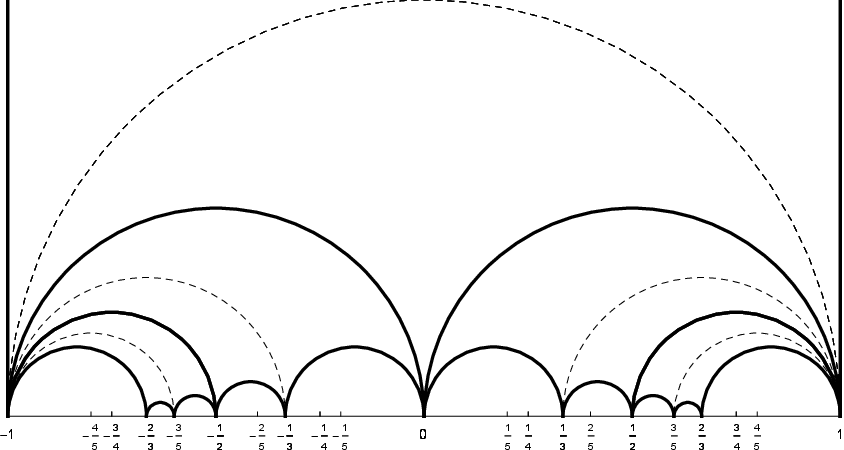}
    \end{subfigure}%
    \hfill\raisebox{5em}{\LARGE{$\dots$}}\hfill
    \begin{subfigure}{.45\textwidth}
        \includegraphics[width=\textwidth]{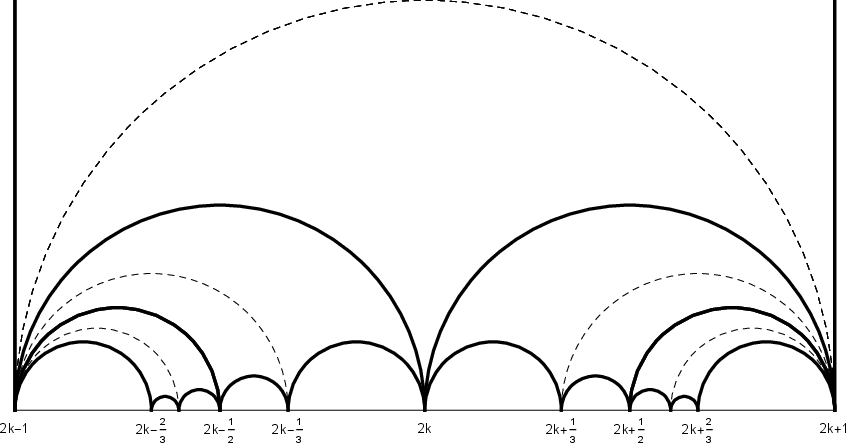}
    \end{subfigure}
    \caption{The tessellation of $\H$ by the ideal triangle $\{-1,1,\infty\}$, where the solid lines represent images of $-1+i\R$ (Type 1 edges) dotted lines represent the images of $e^{it}$ (Type 2 edges) under $\Theta$ acting by M\"obius transformations.}
    \label{new_tes}
\end{figure}
Let  $X_e$ be the set of elements $(\pi_e(\x),u_\gamma)\in T^1\MM_e$ with base point $\pi_e(\x)$ on the line $\pi_e(\pm1+i\R)$ and 
unit tangent vector $u_\gamma$ pointing along $\pi_e(\gamma)$ such that $\pi_e(\y)$ gives the base point of the first return of $\pi_e(\gamma)$ to $X_e$. Using the same coding as in \cite{Ser}, the base point $\x$ breaks the cutting sequence of $\pi_e(\gamma)$ into strings $L^{2k-2}R, L^{2k-1}R, R^{2k-2}L, R^{2k-1}L$ that give the digits of the even continued fraction expansion of $\gplus$. However, $\MM_e$ has two cusps, with $\pi_e(\pm1+i\R)$ the geodesic between the two cusps, so the relationship between the cutting sequence of geodesics on $\H$ and the cutting sequence for geodesics on $\MM_e$ is less clear.

We modify the Farey tessellation to clarify the relationship between oriented geodesics and the $\Theta\infty$ cusp. First, we remove the edges of the Farey tessellation connecting two elements of $\Theta\infty=\{\tfrac{m}{n}\in\Q:m\not\equiv n \pmod 2\}.$ For example, we remove all $2k+i\R$ and $2k+\tfrac{1}{4}+\tfrac{1}{4}e^{it}$. The remaining edges of the Farey tessellation are in bold in Figure \ref{new_tes} and connect an element of the $\Theta\infty$ orbit to an element of the orbit $\Theta 1=\{\tfrac{m}{n}\in\Q:m,n\in2\Z+1\}$, which we will call Type 1 edges. Note that these edges project to the geodesic between the two cusps on $\MM_e$. For geodesics $\gamma$ on $\H$. The dotted line edges of the tessellation connect two elements of the orbit $\Theta 1$, and we will call these Type 2 edges. These edges project to the singular line which runs from the $\pi_e(1)$ cusp to $\pi_e(i)$ and back, which we will call Type 2 edges.  We will label segments of geodesics cut by two successive Type 1 edges. We will use the example of $\gplus>0$ to help illustrate the definitions. The solid blue line in Figure \ref{ecf_gauss_ex} is an example of $\gplus>1$, and the dotted blue line is an example of $\gplus<-1$.

\begin{figure}
 \includegraphics[width=\textwidth]{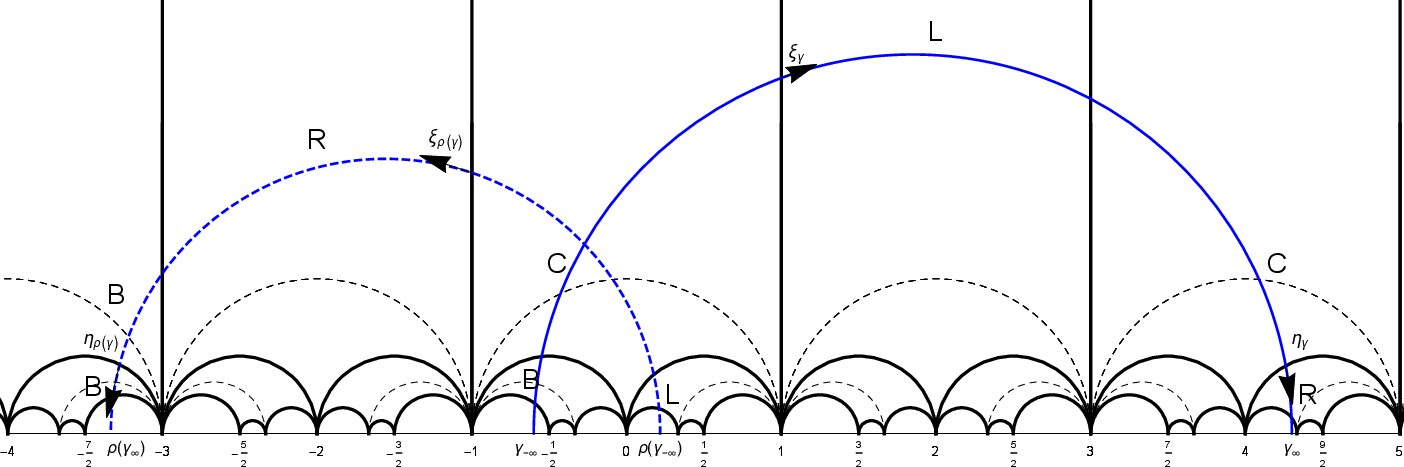}
 \caption{Geodesics $\gamma$ (solid) and $\rho(\gamma)$ (dashed) with segments between two Type 1 edges labeled. Segments are labeled $\mathbf{L}$ or $\mathbf{R}$ if the segment does not cross a Type 2 edge. Segments that cross a Type 2 edge are labeled $\mathbf{B}$ if the Type 1 edges meet at a vertex and $\mathbf{C}$ if they do not.}
 \label{ecf_gauss_ex}
\end{figure}

 We know that the first digit of $\gplus>1$ is $(2k, +1)$ if $\gplus\in[2k-1,2k)$ and is $(2k, -1)$ if $\gplus\in[2k,2k+1)$. Thus, $\gamma$ is cut by $k-1$ Type 1 edges after $1+i\R$ before hitting the Type 2 edge $2k+e^{it}$. As with in the original coding, we note that the Type 1 edges meet at a vertex on the left of the geodesic and label the edge $\mathbf{L}$. We will use a bold letter to differentiate it from the original coding.
Thus, there are $k-1$ segments labeled $\mathbf{L}$ between $\x$ and $\y$, corresponding to the fact that $\gplus=2k\pm w$ for some $w\in[0,1]$. Note that the original cutting sequence for this segment is $L^{2k-2}$, and on $\MM_e$, the $\pi_e(\infty)$ cusp is on the left of $\bar\gamma$. For segments that are cut by two successive Type 1 edges with a vertex on the right and do not cross a Type 2 edge, we label the segment $\mathbf{R}$, corresponding to the cutting sequence $R^2$ in the Farey tessellation coding.

 Now, if $\gplus\in[2k-1,2k)$, $\gamma$ is cut by $2k-1+i\R, 2k+e^{it},$ and $2k-\frac{1}{2}+\frac{1}{2}e^{it},$ which all meet at the vertex $2k-1$. We label this segment $\mathbf{B}$, and do the same for all segments that cross a Type 2 edge between two successive Type 1 edges that meet at a vertex. On $\MM_e$, this corresponds to $\bar\gamma$ crossing $\pi_e(e^{it})$ without going around the cone. 
 Note that $\mathbf{L^{k-1}B}$ corresponds to the Farey tessellation cutting sequence $L^{2k-2}R$, where $R$ means that the $\pi_e(1)$ cusp is on the right of the geodesic. However, here we are keeping track of the $\pi_e(\infty)$ cusp, which remains on the left of the geodesic.

 Finally, if $\gplus\in[2k,2k+1)$, $\gamma$ is cut by $2k-1+i\R, 2k+e^{it},$ and $2k+\frac{1}{2}+\frac{1}{2}e^{it}.$ We will label all segments that cross a Type 2 edge between two successive Type 1 edges that do \emph{not} meet at a vertex $\mathbf{C}$. 
 When $\gplus>1$,  $\mathbf{L^{k-1}C}$ corresponds to the Farey tessellation cutting sequence $L^{2k-2}LR$. On $\MM_e$, the $\pi_e(\infty)$ cusp is on the left of the corresponding segment of $\bar\gamma$, which then crosses the line that runs between $\pi_e(i)$ and $\pi_e(\infty)$, wraps around the $\pi_e(i)$ cone point, and then the $\pi_e(\infty)$ cusp is on the right of the corresponding segment of $\bar\gamma$. 
 That is, the $\pi_e(\infty)$ is on the left of $\bar\gamma$ at the start of the corresponding segment, and on the right at the end.

\begin{figure}
 \includegraphics[width=.6\textwidth]{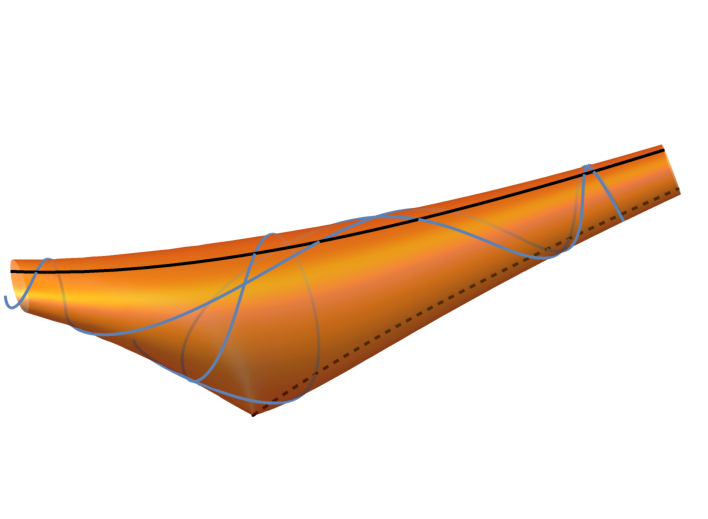}
  \caption{Geodesic on the modular surface $\MM_e$, shown in blue, with the image of $1+i\R$ shown in black  and the image of the $e^{it}$ as a dashed line. These two lines meet at $\pi_e(\infty)$, and the modified cutting sequence keeps track of whether the cusp $\pi_e(\infty)$ is on the left or right.}
\end{figure}

Define $\rho_e:\SS_e\to\SS_e$ by 
$(x,y)\mapsto \left(\frac{1}{ 2k\e-x},\frac{1}{2k\e-y}\right)$ for 
$\e=\operatorname{sign}(x), |x|\in[2k-1,2k+1)$.
As in Section \ref{series}, we connect the action of $\rho_{e}$ with the cutting sequence associated to $\gamma$.

\begin{prop}[Sections 6.2 and 6.3 of \cite{BM1}]
 The $k^{th}$ digit of the even continued fraction of $\gplus$ is given by the cutting sequence of $\gamma$ between $\xi_{\rho_e^{k-1}(\gamma)}$ and $\eta_{\rho_e^{k-1}(\gamma)}$, using the following rules: 
 \begin{center}
 \begin{tabular}{rccc}
        &Original cutting sequence & New cutting sequence & Digit \\\hline
    $\gplus>0$: & $L^{2k-2}R$ & $\mathbf{L^{k-1}B}$ & $(2k,-1)$\\
    & $L^{2k-1}R$ & $\mathbf{L^{k-1}C}$ & $(2k,+1)$\\ \hline
    $\gplus<0$: & $R^{2k-2}L$ & $\mathbf{R^{k-1}B}$ & $(2k,-1)$\\
    & $R^{2k-1}L$ & $\mathbf{R^{k-1}C}$ & $(2k,+1)$\\ \hline
 \end{tabular}
\end{center}
 
 Similarly, the $k^{th}$ digit of the extended even continued fraction expansion of $\gneg$ is given by the cutting sequence of $\gamma$ between  $\xi_{\rho_e^{-k}(\gamma)}$ and $\eta_{\rho_e^{-k}(\gamma)}$.  Here, $\mathbf{L^{k-1}B}$ and $\mathbf{R^{k-1}B}$ (or $L^{2k-2} R $ and $ R^{2k-2} L$ in the Farey tessellation cutting sequence) give the digit $(-1/2k)$ and $\mathbf{L^{k-1}C}$ and $\mathbf{R^{k-1}C}$
 ($L^{2k-1} R $ and $R^{2k-1} L$) give the digit $(+1/2k)$ for $k\geq 1$. 
\end{prop}

\begin{prop}
 Define $\tilde \Omega_e=[0,1]\times[-1,1]\times\{\pm1\}$ and $\tilde T_e:\tilde\Omega_e\to\tilde\Omega_e$ by
\[\tilde T_e(x,y,\e)=\big(\overline T_e(x,y),-\e_1(x)\e\big),
\]
with $\overline T_e$ as in equality \eqref{tebar}. The map $J_e:\SS_e \rightarrow \tOmega_e$, $J_e (x,y)=\operatorname{sign} (x)
(\frac{1}{x},-y,1)$ is invertible. Direct verification reveals the equality \begin{equation}\label{eq6.2} 
J_e \rho_e J_e^{-1} = \tT_e .
\end{equation}
\end{prop}

As before, the push-forwards of the measure $\frac{d\alpha d\beta d\theta}{(\alpha-\beta)^2}$ on $T^1\MM_e$ under the map
$\pi\circ J_e$, is 
$\overline{T}_e$-invariant, where $\pi(x,y,\e)=(x,y)$. We get that the invariant measure for $\overline{T_e}$ is $\frac{dxdy}{(1+xy)^2}$ as in the regular continued fraction case. 

Note that the ideal quadrilateral $\{-1,0,1,\infty\}$ is a threefold cover of the fundamental Dirichlet region $\{z\in\H: |\Re z|\leq\frac{1}{2},|z-1|\geq 1, |z+1|\geq 1\}$ for $\Z_3\ast\Z_3\curvearrowright \H$, which Boca and the author used to describe the odd continued fractions in \cite{BM1}. Thus, the new cutting sequence allows us to massively simplify that cutting sequence description in the same manner.

\section{Even continued fraction slow down map}\label{ecf_farey}

\begin{figure}
 \includegraphics[width=.7\textwidth]{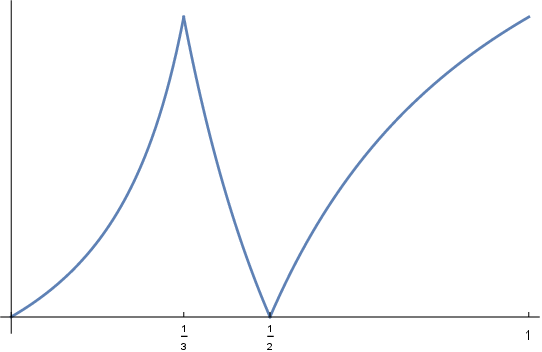}
 \caption{The even Farey map is a slowdown map for the even Gauss map.}\label{efareymap}
\end{figure}

The \emph{even Farey} or \emph{Romik map} $F_e:[0,1]\to[0,1],$ 
\[F_e(x)=
\begin{cases}
\frac{x}{1-2x}& x\in\left[0,\frac{1}{3}\right),\\
\frac{1-2x}{x}& x\in\left[\frac{1}{3},\frac{1}{2}\right),\\
\frac{2x-1}{x}& x\in\left[\frac{1}{2},1\right]
\end{cases}\] shown in Figure \ref{efareymap} is a slowdown of the even Gauss map. That is, for every $x$ there exists an $n\in\N$ such that $F_e^n(x)=T_e$. In this case, we find that $n=k$ for $x\in\left[\frac{1}{2k+1},\frac{1}{2k-1}\right)$. On the even continued fraction expansion of $x=[\![(a_1,\e_1) (a_2,\e_2)\ldots ]\!]$, we find that 
\[F_e([\![(a_1,\e_1) (a_2,\e_2)\ldots ]\!])=
\begin{cases}
[\![(a_1-2,\e_1) (a_2,\e_2)\ldots ]\!]& a_1>2,\\
[\![ (a_2,\e_2)\ldots ]\!]& a_1=2.
\end{cases}\] 
This map was shown to be ergodic and a section of the geodesic flow of the three horned sphere by \cite{AaD, Rom}. Romik \cite{Rom} also used this map to generate Pythagorean triples. 

The natural extension the even Farey map is $\overline{F_e}:[0,1]^2 \to[0,1]^2,$ 
\begin{equation}\label{fe_ext} (x,y)\mapsto 
    \begin{cases}
        \left(\frac{x}{1-2x},\frac{y}{1+2y}\right)
            & x\in\left[0,\frac{1}{3}\right),\\
        \left(\frac{1-2x}{x},\frac{1}{2+y}\right)
            & x\in\left[\frac{1}{3},\frac{1}{2}\right),\\
        \left(\frac{2x-1}{x},\frac{1}{2-y}\right)
            & x\in\left[\frac{1}{2},1\right]
    \end{cases}
\end{equation}
Then $\overline{F_e}^{k}(x,y)=\overline{T_e}(x,y)$, so $\overline{F_e}$ is a slowdown map of $\overline{T_e}.$ Looking at the continued fraction expansion of $x=[\![(a_1,\e_1) (a_2,\e_2)\ldots ]\!],y=\langle\! \langle (a_0,\e_0) (a_{-1},\e_{-1})\ldots \rangle\! \rangle$ we find that 
\begin{align*}
	\overline{F}([\![(a_1,\e_1) (a_2,\e_2)\ldots ]\!],&\langle\! \langle (a_0,\e_0) (a_{-1},\e_{-1})\ldots \rangle\! \rangle)\\*
	&=
	\begin{cases}
		([\![(a_1-2,\e_1) (a_2,\e_2)\ldots ]\!],\langle\! \langle (a_0+2,\e_0) (a_{-1},\e_{-1})\ldots \rangle\! \rangle& a_1>2,\\
		([\![(a_2,\e_2)\ldots ]\!],\langle\! \langle (a_1,\e_1) (a_0,\e_0) (a_{-1},\e_{-1})\ldots \rangle\! \rangle)& a_1=2.
	\end{cases}    
\end{align*}

As with the regular Farey map, we remove some restrictions on the set of geodesics. Consider the set $\wA_e$ of geodesics $\gamma$ on $\H$ with endpoints $(\gplus,\gneg)$ in
\begin{equation}
 \wS_e=\pm\left((1,\infty)\times(-\infty,1)\right).
\end{equation}
All of these geodesics cross either $1+i\R$ or $-1+i\R$. We define $\x$ to be the point where $\gamma$ crosses $\pm1+i\R$. 

We define a new cross section $Y_e$ of unit tangent vectors based on $\pi(\pm1+i\R)$, the geodesic between the cusps, that point along geodesics in $\wA_e$. We define $\y$ to be the base point of the first return to $Y_e$ under the geodesic flow. In the covering space $\H$, $\y$ lifts to next time $\gamma$ crosses a Type 1 side after $\x$. Since the unit tangent vector based at $\y$ and pointing along $\gamma$ is in $Y_e$, and uniquely defines $\gamma$, we may identify the vector with its base point.

Next, we modify the action on $\SS_e$ for $\wS_e$. Define $\wRhoe$ by: 
\[  (x,y)\mapsto
\begin{cases}
(x-2\epsilon,y-2\epsilon), & \epsilon x\in(3,\infty)\\
\left(\frac{1}{2\epsilon - x}, \frac{1}{2\epsilon - y}\right), & \epsilon x\in(1,3],
\end{cases}
\]
where $\epsilon=\operatorname{sign}(x)$. As in the previous cases, $\wRhoe$ induces a map on $\wA_e$ and the first return map on $Y_e$. Next, we consider the map $\wRhoe$ induces on the cutting sequence for $\gamma$, as well as the even continued fraction expansion of the endpoints. Figure \ref{ecf_farey_ex} shows a geodesic and its image under two iterations of $\sigma_e$.

We should also consider the extended even continued fraction expansion of $|\gneg|>1$. Let us consider $\e=-1$. Then $\gneg\in[2k-1,2k+1)$ for $k\geq 0$. When $k=0$, $\gneg\in[-1,1)$ as in the Gauss case. When $k>0$, $\gamma$ crosses $2k-1+i\R,\dots,1+i\R,-1+i\R$, giving the cutting sequence $\mathbf{R^{k}}$. Since cutting sequences are preserved by covering transformations, the cutting sequence for $\gneg-2k=\cfrac{\e_{-2}}{a_{-2}+\dots}\in[-1,1)$ is $\dots \mathbf{L^{j-1}C}\x\dots$ when $\e_{-2}=+1,\frac{a_{-2}}{2}=j$ and $\dots \mathbf{R^{j-1}B}\x\dots$ when $\e_{-2}=-1,\frac{a_{-2}}{2}=j$. Thus, the cutting sequences $\dots \mathbf{L^{j-1}CR^{k}}\x\dots$ and $\dots \mathbf{L^{j-1}C R^{k-1}}\x\dots$ correspond to $\gneg=2k+\frac{e_{-2}}{2j+\dots}$. The $\e=+1$ case is analogous, with $\mathbf{L}$ and $\mathbf{R}$ switched.

\begin{figure}
 \includegraphics[width=\textwidth]{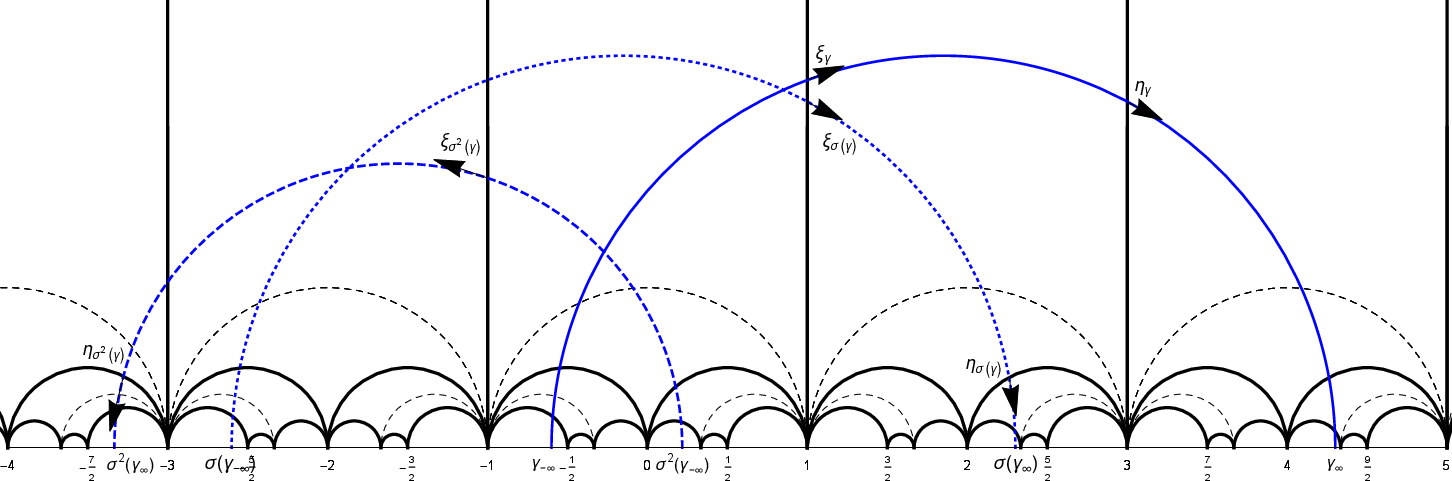}
 \caption{Geodesic $\gamma$ (solid), $\sigma_e(\gamma)$ (dotted), and $\sigma_e^2(\gamma)=\rho_e(\gamma)$ (dashed). The cutting sequence for $\gamma$ is $\dots \mathbf{BC\x L\y CRBB}\dots$, 
 for $\sigma_e(\gamma)$ is $\dots \mathbf{BCL}\xi_{\sigma_e(\gamma)}\mathbf{C} \eta_{\sigma_e(\gamma)} \mathbf{RBB}\dots$, and 
 for $\sigma_e^2\gamma$ is $\dots \mathbf{BCL C} \xi_{\sigma_e^2(\gamma)}\mathbf{RBB}\eta_{\sigma_e^2(\gamma)}\mathbf{BB}\dots$}
 \label{ecf_farey_ex}
\end{figure}

Now we will consider the effect of $\sigma_e$ on both the cutting sequence of $\gamma$ and the continued fraction expansions of $(\gplus,\gneg)$. As we are more interested in $\gplus,$ we will focus on the $\e=+1$ case when necessary to differentiate between $\e=+1$ and $\e=-1$. In order to avoid fractions in the exponents, let $k_i=\frac{a_i}{2}$.

\begin{description}[style=unboxed,leftmargin=.5cm]
    \item[Case 1, $\e\gplus >3$] In this case, 
    \[\raisebox{1em}{$\gplus=\e$}\left(\raisebox{1em}{$a_0+\cfrac{\e_0}{a_1+\cfrac{\e_1}{a_2+\dots}}$}\right)\raisebox{1em}{$,\
    \gneg=-\e$}\left(\raisebox{1em}{$a_{-1}+\cfrac{\e_{-2}}{a_{-2}+\cfrac{\e_{-3}}{a_{-3}+\dots}}$}\right).\]
    Then $\wRhoe$ is orientation preserving and
    \[\raisebox{1em}{$\wRhoe(\gplus)=\e$}\left(\raisebox{1em}{$a_0-2+\cfrac{\e_0}{a_1+\cfrac{\e_1}{a_2+\dots}}$}\right)\raisebox{1em}{$,
    \ \wRhoe(\
    \gneg)=-\e\left({a_{-1}+2+\cfrac{\e_{-2}}{a_{-2}+\dots}}\right).$}\]

    When $\e=+1,$ $\gamma$ has cutting sequence $\dots \mathbf{L^{k_{-1}}}\x \mathbf{L} \y \mathbf{L^{k_0-2}}\dots$ for $k_0\geq2$.
    When  $k_0>2$, $\wRhoe(\gamma)$ has cutting sequence 
    $\dots \mathbf{L^{k_{-1}+1}}\xi_{\wRhoe(\gamma)} \mathbf{L}\eta_{\wRhoe(\gamma)} \mathbf{L^{k-3}}\dots$ 
    with $\gplus>3$. 
    When $k_0=2$, $\wRhoe(\gamma)$ has cutting sequence 
    $\dots \mathbf{L^{k_{-1}+1}}\xi_{\wRhoe(\gamma)} \mathbf{L}\eta_{\wRhoe(\gamma)} \mathbf{B}\dots$ or $\mathbf{L^{k_{-1}+1}}\xi_{\wRhoe(\gamma)} \mathbf{L}\eta_{\wRhoe(\gamma)} \mathbf{C}\dots$, and $1<\gplus<3$. 
    When $\e=-1,$ the $\mathbf{L}$ are replaced with $\mathbf{R}$.

    \item[Case 2, $2<|\gplus|\leq3$] Now 
    \begin{align*}\raisebox{1em}{$
        \gplus=\e$}&\left(\raisebox{1em}{$2+\cfrac{1}{a_1+\cfrac{\e_1}{a_2+\dots}}$}\right)\raisebox{1em}{$,\
    \gneg=-\e$}\left(\raisebox{1em}{$a_{-1}+\cfrac{\e_{-2}}{a_{-2}+\cfrac{\e_{-3}}{a_{-3}+\dots}}$}\right),\\
    \wRhoe(\gplus)=-&\e\left(a_1+\cfrac{\e_1}{a_2+\dots}\right),\wRhoe(\
    \gneg)=\e\cfrac{1}{a_{-1}+2+\cfrac{\e_{-2}}{a_{-2}+\dots}},\end{align*} 
    and $\wRhoe$ is orientation reversing.

    When $\e=+1,$  $\gamma$ has cutting sequence $\dots\mathbf{L^{k_{-1}}}\x \mathbf{C} \y \mathbf{R^{k_1-1}}\dots$ for $k_1\geq1$.  When $k_1>1,$ $\wRhoe(\gamma)$ has cutting sequence $\dots\mathbf{L^{k_{-1}}C}\x \mathbf{R} \y \mathbf{R^{k_1-2}}\dots$ and $\gplus<-3$.
    When $k_1=1$, $\wRhoe(\gamma)$ has cutting sequence $\dots\mathbf{L^{k_{-1}}C}\x \mathbf{B} \y \dots$ or $\dots\mathbf{L^{k_{-1}}C}\x \mathbf{C} \y \dots$, and $-3\leq\gplus<-1$. 
     When $\e=-1,$ the roles of $\mathbf{L}$ and $\mathbf{R}$ are reversed.
    
    \item[Case 3, $1<|\gplus|\leq2$] Now 
    \begin{align*}\raisebox{1em}{$
        \gplus=\e$}&\left(\raisebox{1em}{$2-\cfrac{1}{a_1+\cfrac{\e_1}{a_2+\dots}}$}\right)\raisebox{1em}{$,\
    \gneg=-\e$}\left(\raisebox{1em}{$a_{-1}+\cfrac{\e_{-2}}{a_{-2}+\cfrac{\e_{-3}}{a_{-3}+\dots}}$}\right),\\
    \wRhoe(\gplus)=&\e\left(a_1+\cfrac{\e_1}{a_2+\dots}\right),\wRhoe(\
    \gneg)=-\e\cfrac{-1}{a_{-1}+2+\cfrac{\e_{-2}}{a_{-2}+\dots}},\end{align*}
    and $\sigma_e$ is orientation preserving.

    Finally, when $\e=+1$, $\gamma$ has cutting sequence $\dots\mathbf{L^{k_{-1}}}\x \mathbf{B} \y \mathbf{L^{k_1-1}}\dots$ for $k_1\geq1$.  When $k_1>1,$ $\wRhoe(\gamma)$ has cutting sequence $\dots\mathbf{L^{k_{-1}}B}\xi_{\wRhoe(\gamma)} \mathbf{L}\eta_{\wRhoe(\gamma)} \mathbf{L^{k_1-2}}\dots$ and $\gplus>3$.
    When $k_1=1$, $\wRhoe(\gamma)$ has cutting sequence $\dots \mathbf{L^{k_{-1}}B}\xi_{\wRhoe(\gamma)} \mathbf{B}\eta_{\wRhoe(\gamma)} \dots$ or $\dots \mathbf{L^{k_{-1}}B}\xi_{\wRhoe(\gamma)} \mathbf{C}\eta_{\wRhoe(\gamma)} \dots$, and $1<\gplus\leq3$. 
    Then 
    \[\wRhoe(\gplus)=\e\left(a_1+\cfrac{\e_1}{a_2+\dots}\right),\wRhoe(\
    \gneg)=-\e\cfrac{-1}{2+\cfrac{\e_{-1}}{a_{-1}+\dots}},\] 
    and $\wRhoe$ is orientation preserving. Again, when $\e=-1$, the $\mathbf{L}$s are replaced with $\mathbf{R}$.
\end{description}


\begin{prop}
    Define $\tilde \Omega_e=[0,1]^2\times\{\pm1\}$ and $\tilde F_e:\tilde\Omega_e\to\tilde\Omega_e$ by
    \[\tilde F_e(x,y,\delta)=
    \begin{cases}
        \big(\overline F_e(x,y),\delta\big), 
            & x\in[0,\frac{1}{3})\cup[\frac{1}{2},1),\\
        \big(\overline F_e(x,y),-\delta\big), 
            & x\in[\frac{1}{3},\frac{1}{2})
    \end{cases}
    \]
    with $\overline F_e$ as in equality \eqref{fe_ext}. The map $K_e:\SS_e \rightarrow \tOmega_e$, $K_e (x,y)=
    (\frac{\e}{x},\frac{1}{2-\e y},\e)$ , where $\e=\operatorname{sign}(x)$, is invertible and 
       $K_e \wRhoe  = \tT_e K_e.$
   Furthermore, the pushforward of the invariant measure for $\wRhoe$ is invariant for $\overline{F_e}$ and given by $\frac{dxdy}{(x+y-2xy)^2}.$
\end{prop}

\begin{proof} We check $K_e^{-1}\wRhoe K_e$ \[(\alpha,\beta)\mapsto \left(\frac{\e}{\alpha},\frac{1}{2-\e\beta },\e\right)\mapsto
    \begin{cases}
        \left(\frac{\e}{\alpha-2\e},\frac{1}{4-\e y},\e\right),
            & 3<\e\alpha,\\
        \left(-2+\e\alpha,\frac{2-\e\beta}{5 -2\e\beta},-\e\right),
            & 2<\e\alpha\leq3,\\
        \left(2-\e\alpha,\frac{2-\e\beta}{3 -2\e\beta},\e\right),
            & 1<\e\alpha\leq2.
    \end{cases}
    \mapsto
    \begin{cases}
        \left(\alpha-2\e,\beta-2\e\right),\\
        \left(\frac{1}{2\e-\alpha},\frac{1}{2\e -\beta},-\e\right),\\
        \left(\frac{1}{2\e-\alpha},\frac{1}{2\e -\beta},\e\right).
    \end{cases}
    \]
    Since $\wRhoe$ is defined piecewise by M\"obius transformations, the invariant density is given by $\frac{d\alpha d\beta}{(\alpha-\beta)^2}$. Let $(\alpha,\beta)=K_e^{-1}(x,y,\e)=\left(\frac{\e}{x},\frac{-\e}{y}+2\e\right)$. Thus, the pushforward measure is \[
    \frac{d(\tfrac{\e}{x})d(\tfrac{-\e}{y}+2\e)}{\left(\tfrac{\e}{x}-(\tfrac{-\e}{y}+2\e)\right)^2}
    =\frac{dxdy}{\left(x+y-2xy\right)^2}.\qedhere
   \]

\end{proof}

\section{Connection between the Farey map and the Farey and Lehner continued fractions}\label{diagram}

Lehner \cite{Leh} introduced the slow continued fractions of the form \[a_0+\cfrac{e_0}{a_1+\cfrac{e_1}{a_2+\dots}},\] where $(a_i,e_i)=(1,+1)$ or $(2,-1)$. Dajani and Kraaikamp \cite{DK2} expanded on this definition, introducing the map $L:[1,2)\to [1,2)$  where \[
    x\mapsto \begin{cases}
        \frac{-1}{x-2}, & x\in \left[1,\frac{3}{2}\right)\\
        \frac{1}{x-1}, & x\in \left[\frac{3}{2},2\right)
    \end{cases}\]
which is conjugate to the Farey map by $x\mapsto x+1$. They also introduced the dual continued fraction expansion, which they call the \emph{Farey expansions}  and the natural extension 
$\LL:[1,2)\times [-1,\infty)\to[1,2)\times[-1,\infty)$ by 
\[\LL(x,y)=\begin{cases}
    \left(\frac{-1}{x-2},\frac{-1}{2+y}\right), 
        &  x \in \left[1,\frac{3}{2}\right)\\
    (\frac{1}{x-1}, \frac{1}{1+y}), & x\in\left[\frac{3}{2}, 2\right).
\end{cases}\]
While not stated in \cite{DK2}, $\LL$ is conjugate to the natural extension of the Farey map via $(x+1,\frac{1}{1-y}-2)$. This conjugacy is represented by the square that forms the bottom  of the ``box" in \eqref{box_diag}.  The conjugacy between the map on $\SS$ and the natural extension of the Farey map, given in Section \ref{farey} gives the square at the back of the ```box.''

\begin{equation}\label{box_diag}
    \begin{tikzcd}[ampersand replacement=\&,column sep=6em,row sep=3em]
        \SS \arrow[rr,"\sigma"] 
        \arrow[rd,"{\left(\frac{1}{x}+\e,\frac{1}{y}+\e\right)}",pos=0.6]
        \arrow[dd, "{\left(\frac{\e}{x},\e \left(\frac{-1}{y}+1\right)\right)}"] 
            \&  \&  
                    \SS \arrow[dd] 
                    \arrow[rd,"{\left(\frac{1}{x}+\e,\frac{1}{y}+\e\right)}",pos=0.7]                                    
                    \\
            \& \SS_L \arrow[dd,"{\left(\e x,-\e y\right)}",pos=0.3,crossing over]
            \arrow[rr,"{\begin{cases}
            \left(\frac{1}{2\e-x},\frac{1}{2\e-y}\right), & 1\leq\e x<\frac{3}{2}\\
            (\frac{1}{\e-x}, \frac{1}{\e-y}), & \frac{3}{2}\leq\e x< 2
            \end{cases}}",pos=.2,crossing over]
                 \&  \&  
                        \SS_L \arrow[dd,"{\left(\e x,-\e y\right)}",pos=0.3]
                                  \\
        {[0,1)^2} \arrow[rd,"{\left(x+1,\frac{1}{1-y}-2\right)}",swap]
        \arrow[rr,"\overline{F}",pos=.7]
        \&   \&
                {[0,1)^2} \arrow[rd, "{\left(x+1,\frac{1}{1-y}-2\right)}"]    
              \\
            \& {[1,2)\times [-1,\infty)} 
            \arrow{rr}
            {\LL} 
            \arrow[d, "{\left(\frac{1}{x},y\right)}"',swap]
                \&  \&
                    {[1,2)\times [-1,\infty)} 
                    \arrow[d,"{\left(\frac{1}{x},y\right)}"',swap] 
         \\
            \& {(\frac{1}{2},1]\times[-1,\infty)} 
            \arrow{rr}[swap]
            {\begin{cases}
                \left(\frac{1}{x}-1,\frac{-1}{2+y}\right), & \frac{1}{2}< x\leq\frac{2}{3}\\
                (\frac{-1}{x}+2, \frac{1}{1+y}), & \frac{2}{3}< x\leq 1
            \end{cases}} 
                \& \& 
                        {(\frac{1}{2},1]\times[-1,\infty)} 
        \end{tikzcd}
    \end{equation}

The author used cutting sequences to describe the Lehner and Farey expansions of the 
\[(\gplus,\gneg)\in\SS_L
=\pm\big(([1,2)\times (-\infty,1]\big)\]
respectively \cite{Mer}. In the same paper, the author described an alternate slow continued fraction expansion on $(\frac{1}{2},1]$, which is also dual to the Farey continued. For $x\in(\frac{1}{2},1],$ 
\[x=\cfrac{1}{a_1+\cfrac{e_1}{a_2+\cfrac{e_2}{a_3+\dots}}},\]
and the Gauss map for this continued fraction expansion $L^\ast : (\frac{1}{2},1] \to (\frac{1}{2},1]$ is given by: 
\[x\mapsto \begin{cases}
\frac{1}{x}-1, & x\in\left(\frac{1}{2},\frac{2}{3}\right],\\
\frac{-1}{x}+2, & x\in\left(\frac{2}{3},1\right].
\end{cases}
\]
By construction, $L^\ast$ is conjugate to $L$ by the map $x\mapsto \frac{1}{x}$ and to the Farey map by $x\mapsto 
\frac{1}{x+1}$. The conjugacies between the maps on $\SS_L, [1,2)\times[-1,\infty),$ and $(\frac{1}{2},1]\times[-1,\infty)$ are given by the front of the ``box'' and lower square in in \eqref{box_diag}.

\subsection*{Acknowledgements} The modular surface in Figure \ref{modboat_figure} was originally created with Steve Trettel while he and the author were at ICERM. The author would also like to thank Florin Boca for originally proposing this problem,  for help in the early stages of the project, and feedback draft.


\bibliographystyle{alpha}
\bibliography{even_farey_flows}
\end{document}